\documentclass[12pt, final]{amsart}
\usepackage[cp850]{inputenc}
\usepackage{amssymb}
\usepackage[mathscr]{eucal}
\usepackage{amscd}
\newcommand{\R}{\mathbb R}
\newcommand{\N}{\mathbb N}

\theoremstyle{plain}
\newtheorem{theorem}{Theorem}

\newtheorem{prop}{Proposition}
\newtheorem{proposition}{Proposition}
\newtheorem{lemma}{Lemma}
\newtheorem{cor}{Corollary}

\theoremstyle{remark}

\newtheorem{defn}[theorem]{Definition}
\newcommand{\adef}{\begin{defn}}
\newcommand{\zdef}{\end{defn}}

\title{On nested sequences of convex sets in a Banach space}
\thanks{The research of the first two authors was realized during a visit to the University of Bologna, supported in part
by project MTM2010-20190. The research of the first author was
supported in part by the program Junta de Extremadura GR10113 IV
Plan Regional I+D+i, Ayudas a Grupos de Investigaci\'on,}
\thanks{}
\author{Jes\'us M. F. Castillo}
\author{Manuel Gonz\'alez }
\author{Pier Luigi Papini}
\address{Departamento de Matem\'aticas, Universidad de Extremadura, Avda de elvas s/n, 06071 Badajoz,
Espa\={n}a; Departamento de Matem\'aticas, Universidad de
Cantabria, Avda de los Castros s/n, 39071 Santander, Espa\={n}a;
via Martucci, 19, 40136 Bologna, Italia}

\begin{document}
\maketitle
\begin{abstract} In this paper we study different aspects of the representation
of weak*-compact convex sets of the bidual $X^{**}$ of a separable
Banach space $X$ via a nested sequence of closed convex bounded
sets of $X$. \end{abstract}

\section{Introduction}

In this paper we solve several problems about nested intersections
of convex closed bounded sets in Banach spaces.

We begin with a study of different aspects of the representation
of weak*-compact convex sets of the bidual $X^{**}$ of a separable
Banach space $X$ via a nested sequence of closed convex bounded
sets of $X$. Precisely, let us say that a convex closed bounded
subset $C\subset X^{**}$ is \emph{representable} if it can be written as
the intersection
$$C = \bigcap_{n\in \N} \overline{C_n}^{w^*}$$
of a nested sequence $(C_n)$ of bounded convex closed subsets of
$X$. This topic was considered in \cite{castpapi,limit}, where the
problem of which weak*-closed convex sets of the bidual are
representable was posed. In \cite{berna}, Bernardes shows that
when $X^*$ is separable every weak*-compact convex subset of
$X^{**}$ is representable. Here we will show that compact convex
sets of $X^{**}$ are representable if and only if $X$ does not
contain $\ell_1$ and also that there are spaces without copies of
$\ell_1$ containing weak*-compact convex metrizable subsets of the bidual
not representable.

In Section 3, we solve problem (2) in \cite{limit} showing that
when the sets are viewed as the distance type (in the sense of
\cite{castpapi}) they define, i.e., as elements of $\R^X$, then
every weak*-compact convex set $C\subset X^{**}$ is represented
by a nested sequence $(C_n)$ of closed convex sets of $X$; which
means that for all $x\in X$
$$\mathrm{dist} (x, C) = \lim \mathrm{dist} (x, C_n).$$

In Section 4 we present two examples: the first one solves
Marino's question \cite{mari} about the possibility of enlarging
nested sequences of convex sets to get better intersections; the
second one solves Behrends' question about the validity of
$\varepsilon=0$ in the Helly-B\'ar\'any theorem \cite{behr}.

\section{Representation of convex sets in biduals}

\adef A Banach space is said to enjoy the Convex Representation
Property, in short CRP (resp. Compact Convex Representation
Property, in short CCRP) if every weak*-compact (resp. compact)
convex subset $C$ of $X^{**}$ can be represented as the
intersection
$$C = \bigcap_{n\in \N} \overline{C_n}^{w^*}$$
of a nested sequence $(C_n)$ of bounded convex closed subsets of
$X$. \zdef

\begin{prop} A separable Banach space has CCRP if and only if it does not contain
$\ell_1$.
\end{prop}
\begin{proof} The necessity follows from the Odell-Rosenthal characterization
\cite{odell} of separable Banach spaces containing $\ell_1$.
Indeed, if $X$ contains $\ell_1$ then there is an element $\mu\in
X^{**}$ which is not the weak*-limit of any sequence of elements
of $X$. Hence, $\{\mu\}=\bigcap_{n \in \N} \overline{C_n}^{w^*}$
is impossible: taking an element $c_n\in C_n$ one would get
$\emptyset \neq \bigcap_n \overline{\{c_k: k\geq n\}}^{w^*} \subset \bigcap_n \overline{C_n}^{w^*}  = \{\mu\}$, which means that $\mu$ is the only weak*-cluster point of the sequence $(c_n)$ and thus $\mu =
w^*-\lim c_n$.

As for the sufficiency, let $K$ be a compact convex subset of
$X^{**}$. For every $n\in \N$, let $F_n=\{z_n^k : k\in I_n \}$ be
a finite subset of $K$ for which $K\subset F_n +
n^{-1}B_{X^{**}}$. There is no loss of generality assuming that
$F_n\subset F_{n+1}$. For each $z_n^k\in F_n$, let
$(x_n^k(m))_m\subset X$ be a sequence in $X$ weak*-convergent to
$z_n^k$. Set

$$C_n = \overline{conv}\{ x_n^k(m) , k\in I_n, m\geq n \}+ n^{-1}B_X$$

It is clear that $C_n$ is a nested sequence of closed convex sets
of $X$. Moreover,

$$K\subset F_n + n^{-1}B_{X^{**}} \subset \overline{\overline{conv}\{ x_n^k(m) , k\in I_n, m\geq n\}+ n^{-1}B_X}^{w^*} =\; \overline{C_n}^{w^*},$$
and thus $K\subset \bigcap_n \overline{C_n}^{w^*}$.

Fix now $p\in \bigcap_n \overline{C_n}^{w^*}$. Since $p\in
\overline{C_n}^{w^*}$, there is a finite convex combination
$\sum_{i\in I_n} \theta_i z_n^i$ for which $\|p- \sum_{i\in I_n}
\theta_i z_n^i\|\leq n^{-1}$. This implies that
$p\in\overline{K}=K$ and thus $\bigcap_n
\overline{C_n}^{w^*}\subset K$.\end{proof}

This shows that Problem 1 in \cite{limit} has a negative answer.
On the other hand, Bernardes obtains in \cite{berna} an
affirmative answer when $X^*$ is separable, which is somehow the
best that can be expected. Let us briefly review and extend
Bernardes' result. Recall that a partially ordered set $\Gamma$ is
called filtering when for any two points $i,j\in \Gamma$ there is
$k\in \Gamma$ such that $i\leq k$ and $j\leq k$. An indexed family
of subsets $(C_\alpha)_{\alpha\in \Gamma}$ will be called
filtering when it is filtering with respect to the natural
(reverse) order; i.e., whenever $\alpha\leq \beta$ then $C_\beta
\subset C_\alpha$. One has:

\begin{proposition}\label{main} If $C$ is a convex weak*-compact set in the bidual
$X^{**}$ of a Banach space $X$ then there is a filtering family
$(C_\alpha)_{\alpha \in \Gamma}$ of  convex bounded and closed
subsets of $X$ such that
$$C = \bigcap_{\alpha\in \Gamma} \overline{C_\alpha}^{w^*}.$$
\end{proposition}
\begin{proof} There is no loss of generality assuming that $C\subset B_{X^{**}}$. Let $\Gamma$ be the partially ordered set of finite
subsets of $B_{X^*}$. For each $\alpha\in \Gamma$ we denote
$|\alpha|$ the cardinal of the set $\alpha$. Set now

$$C_\alpha = \{ x\in X: \exists z\in C : \forall y\in \alpha:
|(z-x)(y)|\leq |\alpha|^{-1}\}.$$

This family $(C_\alpha)_{\alpha \in \Gamma}$ is filtering, as well
as $(\overline{C_\alpha}^{w^*})_{\alpha \in \Gamma}$, which
ensures that $\bigcap_{\alpha\in \Gamma}
\overline{C_\alpha}^{w^*}$ is nonempty. Let us show the equality
$$C = \bigcap_{\alpha\in \Gamma} \overline{C_\alpha}^{w^*}.$$

\begin{itemize}

\item $C\subset \bigcap_{\alpha\in \Gamma}
\overline{C_\alpha}^{w^*}$: Let $z\in C \subset B_{X^{**}}$; given
$\alpha\in \Gamma$, by the Banach-Alaoglu theorem, there is $x\in
B_X$ such that $|(z-x)(y)|<|\alpha|^{-1}$ for all $y\in \alpha$.
Hence $x\in C_\alpha$ and thus $z\in \overline{C_\alpha}^{w^*}$.

\item $\bigcap_{\alpha\in \Gamma} \overline{C_\alpha}^{w^*}
\subset C$: Let $z\in \bigcap_{\alpha\in
\Gamma}\overline{C_\alpha}^{w^*}$ and let $V_{\alpha,
\varepsilon}$ be the weak*-neighborhood of $0$ determined by
$\alpha\in \Gamma$ and $\varepsilon>0$; i.e., $V_{\alpha,
\varepsilon} = \{p\in X^{**} : \forall y\in \alpha: |p(y)|\leq
\varepsilon$. Pick $\beta\in \Gamma$ with $\alpha\leq \beta$ and
$|\beta|^{-1}\leq \varepsilon$. Since $z\in
\overline{C_\beta}^{w^*}$, there is $x\in C_\beta$ such that
$|(z-x)(y)|\leq  \varepsilon$ for all $y\in \alpha$; which
moreover means that there is some $z'\in C$ such that
$|(z'-x)(y)|\leq |\beta|^{-1} \leq \varepsilon $ for all $y\in
\beta$. Putting all together one gets that for $y\in \alpha$
$$|(z-z')(y)| =|(z-x)(y) +(x-z')(y)| \leq 2\varepsilon$$
and thus $z-z'\in V_{\alpha, 2\varepsilon}$. Hence $z\in
\overline{C}^{w^*} = C$.
\end{itemize}
\end{proof}

The size of $\Gamma$ can be reduced just taking first a dense
subset $Y\subset B_{X^*}$ and then fixing as $\Gamma$ a
fundamental family of finite sets of $Y$, in the sense that every
finite subset of $Y$ is contained in some element of $\Gamma$.
Such reduction modifies the proof as follows: from the first
finite set $\alpha$ --no longer in $\Gamma$-- determining
$V_{\alpha, \varepsilon}$ one must take a set $\beta\in \Gamma$
such that for each $y\in \alpha$ there is $y'\in \beta$ so that
$\|y-y'\|\leq |\beta|^{-1} \leq \varepsilon$. Get $x$ and $z'$ as
above. Finally, for $y\in \alpha$, one gets
$$|(z-z')(y)| =|(z-z')(y-y') +(z-z')(y')| \leq \varepsilon +2\varepsilon = 3\varepsilon.$$

The consequence of such simplification is that when $X^*$ is
separable then $\Gamma$ reduces to $\N$ and thus one gets the main
result in \cite{berna}:

\begin{cor}[Bernardes] Every Banach space with separable dual has CRP.
\end{cor}

One therefore has:$$ X^* \;\; \mathrm{separable} \Longrightarrow
CRP \Longrightarrow CCRP \Longleftrightarrow \ell_1 \nsubseteq X.
$$

This suggests two questions: 1) whether CCRP implies CRP and 2)
whether CRP implies having separable dual. One has

\begin{prop} CCRP does not imply CRP.
\end{prop}

To prove this we are going to show that the James-Tree space
--perhaps the simplest space not containing $\ell_1$ but having
nonseparable dual-- fails CRP. For information about $JT$, we
refer to \cite[Chapter VIII]{vanDulst}. We begin with a
preparatory lemma that can be considered as a complement to Kalton
\cite[Lemma 5.1]{kalt}.

\begin{lemma}\label{kalty} Let $(C_n)_n$ be a nested sequence of bounded closed
convex subsets of a Banach space $X$. If  $\bigcap_n
\overline{C_n}^{w^*}$ is weak*-metrizable then:
\begin{enumerate}
\item Every $g\in \bigcap_n \overline{C_n}^{w^*}$ is the
weak*-limit of a sequence $(c_n)$ with $c_n\in C_n$. \item Every
sequence $(c_n)$ with $c_n\in C_n$ admits a weak*-convergent
subsequence.
\end{enumerate}
\end{lemma}
\begin{proof} (1) is clear: let $(V_n)_n$ be a sequence of
weak*-neighborhoods of $g$ such that $\{g\} = \cap V_n \cap
\bigcap_n \overline{C_n}^{w^*}$. Picking $c_n \in C_n \cap V_n$
one gets $\{g\}=\overline{\{c_n\}}^{w^*}$.

To prove (2), let us consider the equivalence relation on the set
$\mathcal P_\infty(\N)$ of infinite subsets of $\N$: $A\sim B$ if
and only if $A$ and $B$ coincide except for a finite set.
Moreover, $K$ will denote the set of all compact subsets of
$\bigcap_n \overline{C_n}^{w^*}$. Given a sequence $(c_n)$ with
$c_n\in C_n$ we define the following map $w: \mathcal
P_\infty(\N)/\sim \;\;\to K$:

$$w([A]) = \cap_k  \overline {\{c_n : n \in A, n > k\}}^{w^*}.$$ The set $\mathcal P_\infty(\N)/\sim$ admits a
natural order as $[A]\leq [B]$ if $A$ is eventually contained in
$B$. This order has the property that for every decreasing
sequence $([A_n])_n$ there is an element $[B]$ with $[B]\leq
[A_n]$ for all $n$. Since $\bigcap_n \overline{C_n}^{w^*}$ is
metrizable, it follows \cite[Sect. 2]{behrarch} that there is $M\in \mathcal
P_\infty (\N)$ on which $w$ is stationary; i.e., $w([C])=w([M])$
for all infinite subsets $C\subset M$. This immediately yields
that $w(\{c_n\}_{n\in M})$ has only one point, and thus
$\{c_n\}_{n\in M}$ is weak*-convergent. \end{proof}

\medskip

Let us denote by $G$ the set of all branches of the dyadic tree
$T$. For each $r\in G$, let $e_r$ denote the corresponding element
of the basis of $\ell_2(G)$ considered as a subspace of $JT^{**}$.
Let $\{e_{k,l} : k \in \N_0, 1 \leq l \leq 2^k\}$ denote the unit
vector basis of $JT$. The action of $e_r$ on $x^*\in JT^*$ is
given by
$$
\langle x^*,e_r\rangle = \lim_{\textrm{along $r$}} \langle
e_{k,l},x^*\rangle.
$$

For each $m\in \N$ we denote by $P_m$ the norm-one projection in
$JT$ defined by $P_m e_{k,l} = e_{k,l}$ if $k\geq m$, and $P_m e_{k,l} = 0$ otherwise. For each $r\in G$ we consider $f_r\in JT^*$ given
by $\langle e_{k,l},f_r\rangle$ equal to $1$ if $(k,l)\in r$, and
equal to $0$ otherwise. Observe that $\langle f_r,e_s\rangle =
\delta_{r,s}$. Let $S = \{s_n : n\in \N\}$ denote a countable
subset of $G$ such that the branches in $S$ include all the nodes
of the tree $T$.\\

\noindent \textbf{Proof of Proposition 2.} Let us show that the
closed unit ball $B$ of $\ell_2(S)$ cannot be represented. Assume
that we can write $B=\cap_{n\in\N} \overline{C_n}^{w^*}$. The set
$B$ is $w^*$-metrizable, because it is the unit ball of a
separable reflexive subspace. By Lemma \ref{kalty}, each vector in
$B$ is the $w^*$-limit of a sequence $(x_n)$ with $x_n \in C_n$.
For each $s \in S$ we select $x^s_n \in C_n$ so that $w^*$-$\lim
x^s_n = e_s$. Note that $\lim_{n}\|(I-P_k)x^s_n\| =0$ for every
$k$ and $s$.
\medskip

We take $t_1\in S$, $t_1\neq s_1$. Also we take $x_1 =
x^{t_1}_{n_1}$ with $|\langle x_1,f_{t_1}\rangle - 1|<2^{-1}$, and
select $(k_1,l_1)\in t_1\setminus s_1$ such that $\|P_{k_1}x_1\|<
2^{-1}$. Next we take $t_2\in S$ with $(k_1,l_1)\in t_2$ and
$t_2\neq s_2$. Also we take $x_2 = x^{t_2}_{n_2}$ with
$\|(I-P_{k_1})x_2\|<2^{-2}$ and $|\langle x_2,f_{t_2}\rangle -
1|<2^{-2}$, and select $(k_2,l_2)\in t_2\setminus s_2$ with $k_2>
k_1$ such that $\|P_{k_2}x_2\|< 2^{-2}$. Proceeding in this way we
obtain a sequence $(x_i)$ that is eventually contained in each
$C_n$ and an ordered sequence of different nodes $(k_i,l_i)$ that
determine a branch $r\in G\setminus S$. Since $JT$ is separable
and contains no copies of $\ell_1$, the sequence $(x_i)$ has a
subsequence that is $w^*$-convergent to some $x^{**}\in JT^{**}$
\cite[First Theorem in p. 215]{Diestel}.  Thus,
$x^{**}\in\cap_{n\in\N} \overline{C_n}^{w^*}$, but $x^{**}\notin
B$ since  $\langle f_r,x^{**}\rangle =1$. \hfill $\square$\\

Proposition 1 thus characterizes the CCRP, while Proposition 2
shows that even when compact convex sets are representable,
arbitrary weak*-metrizable convex bounded closed sets do not have
to be. The question of which convex sets are representable thus
arises. Bigger than compact spaces are the so called small sets
\cite{beka,small,arias}, but it was shown in \cite{beka} that a
closed bounded convex small set is compact.

\section{Representation of convex sets in the hyperspace}

The theory of types in Banach spaces represents the elements of a
Banach space $g\in X$ as functions $\tau_g(x)= \|x -g\|$. These
are the elementary types and the types are the closure of the set
of elementary types in $\R^X$. It can be shown that bidual types,
i.e., functions having the form $\tau_g(x)= \|x -g\|$ for $g\in
X^{**}$ are also types \cite{farm}. In close parallelism, the theory of  distance types was developed in
\cite{castpapi}: in it, the elements to be represented are the closed bounded convex
subsets $C$ of $X$ via the function $d_C(x)=\mathrm{dist}(x, C)$.
These are the elementary distance types. The $\emptyset$-distance
types are the functions of the form $d(x) = \lim d_{C_n}(x)$ where
$(C_n)$ is a nested sequence of closed bounded convex subsets of
$X$ with empty intersection. In \cite[Thm. 4.1]{castpapi} it was
shown the existence of $\emptyset$-distance types that are not
types in every nonreflexive separable Banach space. It was also
shown \cite[Thm. 5.1]{castpapi} that bidual types on separable
Banach spaces coincide with $\emptyset$-distance types defined by
"flat" (in the sense of Milman and Milman \cite{milmil}) nested
sequences of bounded convex closed sets $(C_n)$. In \cite[Thm.
1]{limit} it is shown that given a nested sequence $(C_n)$ of
bounded convex closed sets on a separable space $X$ one always has
$$\mathrm{dist}(x, \bigcap \overline{C_n}^{w^*}) = \lim \mathrm
{dist} (x, C_n).$$ While Bernardes shows in \cite[Thm. 1]{berna}
that that happens in all Banach spaces.

All this suggests the problem \cite[Problem 2]{limit} whether the
analogue of Farmaki's (bidual types are types) also holds for
distance types; i.e., if given a weak*-compact convex subset $C$ of
$X^{**}$, the \emph{bidual distance type} it defines $d_C(x)=\mathrm{dist}(x, C)$ on
$X$, is a $\emptyset$-distance type. Let us give an affirmative answer.

\begin{prop} Let $C$ be a weak*-compact convex subset of the
bidual $X^{**}$ of a separable space $X$ such that $C\cap
X=\emptyset$. There is a nested sequence $(C_n)$ of closed convex
sets of $X$ such that $C\subset \bigcap_n \overline{C_n}^{w^*}$
and for all $x\in X$
$$\mathrm{dist} (x, C) = \lim \mathrm{dist} (x, C_n).$$
\end{prop}
\begin{proof} Let $(x_n)$ be a dense sequence in $X$. Since $C$ is bounded, it is contained in
the ball $\gamma B_{X^{**}}$ for some $\gamma >0$. We proceed inductively: pick $x_1$,
let $\alpha_1=\mathrm{dist}(x_1, C)$, then set a monotone increasing
sequence $(\alpha_n^1)$ convergent to $\alpha_1$. Pick functionals
$\varphi_n^1\in B_{X^*}$ that strictly separate $C$ and
$x_1+(\alpha_n^1)B_{X^{**}}$, say $$\inf_{z\in C}z(\varphi_n^1) >
\|x_1\|+\alpha_n^1 + 2\varepsilon_n^1.$$ Set $C_{n,1} = \{ x\in X:
\exists z\in C :\; |(z - x)(\varphi_n^1)| \leq \varepsilon_n^1\}
\cap \gamma B_{X^{**}}$. The sequence of convex sets $C_{n,1}$ is
nested and every point $z\in C$ belongs to the weak*-closure of
some set $\{ x\in X:  |(z - x)(\varphi_n^1)| \leq n^{-1}\}$
which is in turn contained in $C_{n,1}$. Thus, $C \subset
\bigcap_n \overline{C_{n,1}}^{w^*}$. Moreover,
$x_1+(\alpha_1)B_{X^{**}} \cap \bigcap_n \overline{C_{n,1}}^{w^*}
=\emptyset$ because otherwise there should be elements $c_n\in
C_{n,1}$ for which $(x_1 + \alpha_1 b - c_n)(\varphi_n^1)<
\varepsilon_n^1$; since there must be $z_n\in C$ for which $|(z_n
- c_n)(\varphi_n^1)|\leq \varepsilon_n^1$, pick $z\in C$ a
weak*-accumulation point of $(z_n)$ to conclude that $(x_1 +
\alpha_1 b - z)(\varphi_n^1)= (x_1 + \alpha_1 b - c_n  + c_n -
z)(\varphi_n^1)\leq 2\varepsilon_n^1$ which immediately yields
$$z(\varphi_n^1) = (x_1 + \alpha_1 b)(\varphi_n^1) - (x_1 + \alpha_1 b - z)(\varphi_n^1) \leq \|x_1\|+\alpha_n^1 + 2\varepsilon_n^1$$
in contradiction with the separation above.

Thus, by \cite[Thm. 1]{limit} we get $\mathrm{dist} (x_1, C) =
\mathrm{dist}(x_1, \bigcap \overline{C_{n,1}}^{w^*}) = \lim
\mathrm{dist} (x_1, C_{n,1}).$

We pass to $x_2$. Everything goes as before except that all the
action is going to happen inside $\bigcap
\overline{C_{n,1}}^{w^*}$. Precisely, once $\alpha_2, \alpha_n^2,
\varphi_n^2, \varepsilon_n^2$ have been fixed by the same
procedure as above, set
$$C_{n,2} = \{ x\in X: \exists z\in C :\; \max_{i=1,2}|(z - x)(\varphi_n^i)| \leq
\varepsilon_n^i\} \cap \gamma B_{X^{**}}$$ to conclude that $C
\subset \bigcap_n \overline{C_{n,2}}^{w^*} \subset \bigcap_n
\overline{C_{n,1}}^{w^*}$ and $\mathrm{dist} (x_i, C) =
\mathrm{dist}(x_i, \bigcap \overline{C_{n,2}}^{w^*}) = \lim
\mathrm{dist} (x_i, C_{n,2})$ for $i=1,2$. Proceed inductively.
Since $C_{n,k+1}\subset C_{n,k}$, we can diagonalize the final
sequence of sequences to get the sequence $(C_{k,k})$, which
satisfies $C\subset \bigcap \overline{C_{k,k}}^{w^*}$ and,
moreover, for all $n$ one has
$$\mathrm{dist} (x_n, C) = \mathrm{dist}(x_i, \bigcap
\overline{C_{k,k}}^{w^*}) = \lim \mathrm{dist} (x_n, C_{k,k}).$$
By continuity, the equality remains valid for
 all $x\in X$.
\end{proof}

In the clasical case, as Farmaki remarks in \cite{farm}, it is not
obvious that \emph{fourth-dual} types, i.e., applications having
the form $\tau_g(x)=\|x +g\|$ for $g\in X^4$ on separable spaces
$X$ are necessarily types. One thus may ask: Let $X$ be a
separable Banach space and let $C \subset X^{2k}$ be a bounded
weak*-closed convex. Must there be a sequence $(C_n)$ of bounded
convex closed subset of $X$ such that for every $x\in X$ one has
$\mathrm{dist}(x, C) = \lim \mathrm{dist}(x, C_n)$?

\section{Further properties of nested sequences}

\subsection{Enlarging sets for better intersection: Marino's
problem} Let $A$ be a closed set. For $\varepsilon>0$ we set

$$A^\varepsilon = \{ x\in X : \mathrm{dist}(x, A)\leq
\varepsilon\}.$$

An extremely nice result of Marino \cite{mari} establishes that
given any family $(G_\gamma)$ of convex sets with nonempty
intersection then either $\bigcap_\gamma G_\gamma^\varepsilon$ is
bounded for every $\varepsilon>0$ or is unbounded for every
$\varepsilon>0$. A question left open in \cite[p.583]{limit} is
whether it is possible to have $\bigcap A_n = \emptyset$, some
intersections $\bigcap A_n^{\varepsilon}$ nonempty and bounded and
others unbounded. The next example shows it can be so:

\noindent \textbf{Example 2.} Consider in $\ell_1$ the sequence
$A_{2k} = \{x\in \ell_1: \;\; x_{k+1}\leq
-\frac{2^{k+1}-1}{2^{k+1}}\}$ and $A_{2k-1} = \{x\in \ell_1: \;\;
x_{k}\geq 1\}$. Then $\bigcap A_n =\emptyset = \bigcap
A_n^\varepsilon =\emptyset $ for all $\varepsilon<1$, while

$$\bigcap A_n^1 = \{x\in \ell_1:\;\; \forall k\;\; 0\leq x_k\leq \frac{1}{2^k}
\}$$

and $\bigcap A_n^{1+\varepsilon}$ is unbounded for all
$\varepsilon>0$ since all $x\in \ell_1$ with $-\varepsilon\leq
x_i\leq 0$ for every $i$ belong to that set.\\

The choice of $\ell_1$ for the example is not at random: during
the proof of \cite[Prop. 9]{limit} it is shown that in reflexive
spaces, $\bigcap A_n = \emptyset$ implies $\bigcap A_n^\varepsilon
= \emptyset$ for all $\varepsilon>0$. Marino's theorem in
combination with \cite[Prop. 9]{limit} yields that in a
non-reflexive space if $\alpha= \inf \{\varepsilon>0: \bigcap
A_n^\varepsilon \neq \emptyset \}$ then either $\bigcap
A_n^\varepsilon$ is bounded for all $\varepsilon>\alpha$ or
unbounded for all $\varepsilon>\alpha$.  Let us show now that
Marino's theorem remains ``almost" valid for nested sequences with
empty intersection in a finite dimensional space. In this case,
the boundedness of some $A_n$ immediately implies, by compactness,
that $\bigcap_{n\in \N} A_n\neq \emptyset$. Assume thus one has a
nested sequence of unbounded convex sets. Let $T_k =\{x\in X:
k\leq \|x\|\leq k+1\}$. One has

\begin{lemma} Let $(A_n)$ be a sequence of unbounded connected
sets in a finite dimensional space $X$. Then either $\bigcap_{n\in
\N} A_n=\emptyset$ or for all but finitely many $k\in \N$ and
every $\varepsilon>0$ there is an infinite subset $N_k\subset \N$
such that $T_k \cap \bigcap_{n\in N_k} A_n^\varepsilon\neq
\emptyset$.
\end{lemma}
\begin{proof} If for every $k\in \N$ the ball $kB$ of radius $k$
does not intersect $\bigcap_{n\in \N} A_n$ then $\bigcap_{n\in \N}
A_n=\emptyset$. Otherwise, let $x_{n,k}\in A_n\cap kB$. Since
$A_n$ is unbounded, there is a point $y_{n, k+1}$ with $\|y_{n,
k+1}\|> k+1$. Since $A_n$ is connected, there is some
$x_{n,k+1}\in A_n$ with $k\leq\|x_{n,k+1}\|\leq k+1$, and thus in
$A_n\cap T_k$. The sequence $(x_{n,k+1})_n$ lies in the compact
set $T_k$ and thus for some infinite subset $N_k\subset \N$ the
subsequence $(x_{n,k+1})_{n\in N_k}$ is convergent to some point
$x_{k+1}\in T_k$. Thus, $x_{k+1} + \varepsilon B$ intersects the
sets $\{A_n: n\in N_k\}$ and thus $\bigcap_{n\in N_k}
A_n^\varepsilon \cap T_k\neq \emptyset$.
\end{proof}

Thus we get:

\begin{proposition} Let $(A_n)$ be a nested sequence of unbounded
connected sets  in a finite dimensional space $X$. Then either
$\bigcap_{n\in \N} A_n=\emptyset$ or $\bigcap_{n\in \N}
A_n^\varepsilon$ is unbounded for every $\varepsilon>0$.
\end{proposition}

The assertion obviously fails for non-connected sets and also fails in infinite dimensional spaces:\\

\noindent \textbf{Example 1.} In $\ell_2$ take $A_n=\{x\in \ell_2:
\forall k>n,\;\; 0\leq x_k\leq 1 \;\;\mathrm{and}\;\; \forall
k\leq n, \;\; x_k=0\}$. This  is a nested sequence of unbounded
convex closed sets such that $\bigcap_{n\in \N} A_n=\emptyset$
while for all $\varepsilon>0$ the set $\bigcap_{n\in \N}
A_n^\varepsilon$ is bounded: indeed, if $y\in A_n^\varepsilon$ for
all $n$ then there is $x_n\in A_n$ for which $\|y - x_n\|\leq
\varepsilon$; thus $\sum_{i=1}^n |y_i|^2 \leq \varepsilon^2$ for
all $n$, so $\|y\|\leq \varepsilon^2$.

\subsection{On the  Helly-B\'ar\'any theorem}

In one of the main theorems in \cite{behr} Behrends establishes a
Helly-B\'ar\'any theorem for separable Banach spaces \cite[Thm.
5.5]{behr}: \emph{Let $X$ be a separable Banach space and
$\mathcal C_n$ a family of nonvoid, closed and convex subsets of
the unit ball $B$ for every $n$. Suppose that there is a positive
$\varepsilon_0 \leq 1$ such that $\bigcap_{C\in \mathcal C_n}
C+\varepsilon B =\emptyset$ for every $n$ and every
$0<\varepsilon<\varepsilon_0$. Then there are $C_n\in \mathcal
C_n$ such that $\bigcap_n C_n+ \varepsilon B =\emptyset$.}
Behrends asks \cite[Remark 2, p. 17]{behr} whether one can put
$\varepsilon=\varepsilon_0$ in the previous theorem. The following
example shows that the answer is no:\\

\noindent \textbf{Example.} In $c_0$, the family $\mathcal C_n$
contains two convex sets:

$$a_n^+ = \{x\in c_0: \forall i\in \N:\; |x_i|\leq \frac{1}{2}(1 +
\frac{1}{i})\;\mathrm{and}\; |x_n| = \frac{1}{2}(1 +
\frac{1}{n})\}$$ and

$$a_n^- = \{x\in c_0: \forall i\in \N:\; |x_i|\leq \frac{1}{2}(1 +
\frac{1}{i})\;\mathrm{and}\; |x_n| = -\frac{1}{2}(1 +
\frac{1}{n})\}.$$

One has $a_n^+ \cap a_n^- =\emptyset$ for all $n\in \N$. But for
every $z\in\{-,+\}^\N$ the choice $a_n^{z(n)}\in \mathcal C_n$ has
$x\in \bigcap_n a_n^{z(n)}\neq \emptyset$ for $x_i =
z(i)\frac{1}{2i}$.

\end{document}